\documentclass[11pt, twoside, a4paper]{article}
\usepackage{
latexsym, amsmath, amsxtra, amssymb, amsthm, exscale, multicol, dsfont, delarray,stmaryrd}
\usepackage[utf8]{inputenc}
\usepackage{verbatim}
\usepackage{fancyhdr}
\usepackage{url}

\usepackage[round]{natbib}

\newcommand{\R}{\mathds{R}}

\newcommand{\E}{\mathds{E}}
\newcommand{\N}{\mathds{N}}

\newcommand{\epsi}{\ensuremath{\varepsilon}}
\newcommand{\skalar}[2]{\left\langle #1,#2 \right\rangle}

\newcommand{\num}[2]{\ensuremath{#1,\ldots,#2}}
\newcommand{\vect}[3]{\ensuremath{#1_{#2},\ldots,#1_{#3}}}

\renewcommand{\rho}{\varrho}

\newcommand{\eqd}{\stackrel{d}{=}}
\newcommand{\eqn}{\begin{equation}}
\newcommand{\nqe}{\end{equation}}
\newcommand{\pmx}{\begin{pmatrix}}
\newcommand{\xmp}{\end{pmatrix}}
\newcommand{\bmx}{\begin{bmatrix}}
\newcommand{\xmb}{\end{bmatrix}}

\DeclareMathOperator{\Var}{Var}

\theoremstyle{definition}
\newtheorem{definition}{Definition}[section]
\theoremstyle{plain}
\newtheorem{satz}[definition]{Theorem}
\newtheorem{lemma}[definition]{Lemma}
\newtheorem{korollar}[definition]{Corollary}
\theoremstyle{remark}
\newtheorem{bem}[definition]{Remark}

\newcommand{\BOX}{\ensuremath\Box}
\renewenvironment{proof}[1]{{\vskip\baselineskip\noindent\textit{Proof#1.}}}%
{\origqed\vskip\baselineskip\gdef\origqed{\hspace*{.1pt}\hspace*{\fill}\BOX}}

\makeatletter
\def\@tagformdelimstart{(}%
\def\@tagformdelimend{)}%
\def\@tagformdel{%
  \gdef\@tagformdelimstart{}%
  \gdef\@tagformdelimend{}%
}
\def\@tagformset{%
  \gdef\@tagformdelimstart{(}%
  \gdef\@tagformdelimend{)}%
}
\def\tagform@#1{%
   \maketag@@@{\@tagformdelimstart\ignorespaces#1\unskip%
   \@@italiccorr\@tagformdelimend}\@tagformset}
\def\origqed{\hspace*{.1pt}\hspace*{\fill}\BOX}
\def\qed{\ifmmode%
 \@tagformdel%
 \tag{\BOX}%
 \else%
 \hspace*{.1pt}\hspace*{\fill}\BOX%
 \fi%
 \gdef\origqed{}}
\makeatother

\allowdisplaybreaks
\title{On tail bounds for random recursive trees}
\author{G\"otz Olaf Munsonius\\[1.2ex]
\small Institute of Mathematics, J.W. Goethe University\\
\small 60054 Frankfurt a.M., Germany\\[1.1ex]
\small munsonius@math.uni-frankfurt.de\normalsize
}

\date{}

\begin{document}

\maketitle

\begin{abstract}
We consider a multivariate distributional recursion of sum-type as arising in the probabilistic analysis of algorithms and random trees.
We prove an upper tail bound for the solution using Chernoff's bounding technique by estimating the Laplace transform.
The problem is traced back to the corresponding problem for binary search trees by stochastic domination.
The result obtained is applied to the internal path length and Wiener index of random $b$-ary recursive trees with weighted edges and random linear recursive trees.
Finally, lower tail bounds for the Wiener index of these trees are given.
\end{abstract}

\noindent%

\section{Introduction}
Many parameters of recursive algorithms, trees or other recursive structures can often be described by a so-called recursion of sum type
\eqn\label{recursion_tail}
X_n\eqd\sum_{i=1}^bA_i(I_n)X_{I_{n,i}}^{(i)}+d(I_n,Z)\qquad(n\ge 2)
\nqe
where $X_n^{(1)},\ldots,X_n^{(b)}$ have the same distribution as $X_n$, $d:\R^b\times\R^b\to\R^k$ and $A_i:\R^b\to\R^{k\times k}$ are deterministic functions, $I_n=(\vect{I}{n,1}{n,b})\in\{\num{0}{n-1}\}^b$ and $Z\in\R_{\ge 0}^b$ are random vectors with $E[d(I_n,Z)]=0$, and $X_n^{(1)},\ldots,X_n^{(b)}$, $I_n$, $Z$ are independent.
By $\eqd$ we denote equality in distribution.

From the algorithmic point of view, such a recurrence arises by considering so-called divide and conquer algorithms.
Let $Y_n$ denote the parameter of interest of the algorithm applied to a problem of size $n$.
The algorithm splits the large problem into $b$ subproblems of the smaller sizes $I_{n,1},\ldots,I_{n,b}$.
If the considered parameter $Y_n$ is essential given by the (possible weighted) sum of the corresponding parameters of the smaller subproblems, for a matrix $C_n$ the vector $X_n:=C_n(Y_n-E[Y_n])$ suffices the recurrence \eqref{recursion_tail} where the coefficients $A_i(I_n)$ are the weights of the subproblems (scaled by $C_n$ and $C_{I_{n,i}}$) and the additional function $d$ gives the cost for splitting the problem in this manner and merging the solutions of the subproblems to a solution of the size $n$ problem.
The vector $Z$ attends more universality.

One famous example for a parameter satisfying recursion \eqref{recursion_tail} is the distribution of the number of comparisons made by quicksort which is equal in distribution to the internal path length of the random binary search tree.
\citet{mcdiarmid_hayward_96} used martingale difference methods to show upper tail bounds for it.
\citet{roesler_92} as well as \citet{fill_janson_01} obtained upper bounds for its Laplace transform by induction.
Having upper bounds for the Laplace transform, they concluded upper bounds for the tails of the distribution by application of Chernoff's bounding technique.
\citet{alikhan_neininger_07} generalized this procedure to the two-dimensional recursion for the Wiener index and the internal path length of the random binary search tree extending their technique in \citet{alikhan_neininger_04} for the analysis of tail bounds for the complexity of a randomized algorithm to evaluate game trees .

In this paper, we apply the method of \citet{alikhan_neininger_07} to multivariate functionals satisfying recursion \eqref{recursion_tail} where the operator norm of the coefficient matrices $A_i$ can be stochastically bounded in a special way.

We denote by $\preceq_{\mathrm{st}}$ the stochastic order and by $U$ a random variable uniformly distributed on $[0,1]$.
The fundamental result of this paper is the following theorem.
\begin{satz}\label{upper_tail_bound}
Let $X_n$ be a solution of the distributional recursion \eqref{recursion_tail}.
Assume that $X_1=0$, $\|d(I_n,Z)\|\le D$ almost surely for all $n\in\N$ and for a constant $D\in\R$ and that
\[
\sum_{i=1}^b\|A_i(I_n)\|_{\mathrm{op}}^2\preceq_{\mathrm{st}}1-U(1-U)
\]
as well as $\|A_i\|_{\mathrm{op}}\le 1$ for all $i\in\{1,\ldots,b\}$.
Let $\gamma\approx 2.0047$ be the positive solution of $12/7=e^{2/\gamma}-2/\gamma$ and $L_0\approx 5.0177$ be the largest root of $e^L=6L^2$.
Then we have for all $t>0$, $n\in\N$ and any component $X_{n,j}$ of $X_n$ ($j\in\{\num{1}{k}\}$) with 
$C:=48D/\gamma+D\sqrt{48(48/\gamma^2-5)}$,
\eqn\label{bounds}
P\left(X_{n,j}>t\right)\le\begin{cases}
\exp\left(-\frac{t^2}{10\gamma^2 D^2}\right),&\text{ if $0\le t\le 5\gamma D$},\\
\exp\left(\frac 52-\frac{t}{\gamma D}\right),&\text{ if $5\gamma D< t\le C$},\\
\exp\left(-\frac{t^2}{96 D^2}\right),&\text{ if $C< t\le 48 DL_0$},\\
\exp\left(24L_0^2-\frac{L_0}{D}t\right),&\text{ if $48 DL_0< t\le 4De^{L_0}$},\\
\exp\left(\frac tD-\frac tD\log\left(\frac{t}{4 D}\right)\right),&\text{ if $4De^{L_0}< t$}.
\end{cases}
\nqe
The same bounds hold for the left tail $P(X_{n,j}<-t)$.
\end{satz}

As an application of Theorem \ref{upper_tail_bound}, we obtain upper tail bounds for the distribution of the internal path length and the Wiener index, in random $b$-ary recursive trees with weighted edges by showing the stochastic domination condition.
The distance between two nodes in a tree is defined as the number of edges on the unique path between the two nodes.
Then, the internal path length of a rooted tree is the sum of all node depths of the tree where the depth of a node is its distance to the root, and the Wiener index is the sum of the distances between all unordered pairs of nodes.

The $b$-ary recursive tree with weighted edges can be considered as a special case of the tree model in the paper of \citet{broutin_devroye_06} in discrete time where the lifetimes of the edges are independent exponentially distributed random variables.
The shape of the random tree is also obtained as an increasing tree due to \citet{bergeron_flajolet_salvy_07} and is a special case of the general model of random trees in \citet{broutin_devroye_mcleish_delasalle_08}.

\begin{satz}\label{upper_tail_bound_wp}
Let $Y_n:=(W_{n},P_n)^T$ denote the vector consisting of the Wiener index and the internal path length of a random $b$-ary recursive tree of size $n$ with edge weights $Z$ where $\|Z\|$ is bounded almost surely.
Then there exists a constant $D$ such that we have in the recursive formula \eqref{recursion_tail} for $X_n$ given by
\[
X_{n}:=\begin{bmatrix} \frac 1{n^2}&0\\0&\frac 1{n}\end{bmatrix} \left(Y_{n}-E[Y_{n}]\right)
\]
almost surely $\|d(I_n,Z)\|\le D$ and the bounds \eqref{bounds} of Theorem \ref{upper_tail_bound} are valid for
\begin{align*}
P\left(\frac{W_{n}-E[W_n]}{n^2}>t\right)&&\text{and}&&P\left(\frac{P_{n}-E[P_{n}]}{n}>t\right)
\end{align*}
as well as for the corresponding left tails $P(X_{n,j}<-t)$ (for $j=1,2$). 
\end{satz}

Using the asymptotic expansion of the expectation of the internal path length and the Wiener index, the following asymptotic tail bounds are obtained.
\begin{korollar}\label{korollar_asymptotischer_tail}~
Let $P_n$ denote the internal path length and $W_n$ be the Wiener index of a random $b$-ary recursive tree of size $n$ with edge weights $Z$ where $\|Z\|$ is bounded almost surely.
Then, there exists a constant $D>0$ such that for $t>0$ and $n\to\infty$ it holds
\begin{align*}
\lefteqn{P(|P_{n}-E[P_{n}]|>tE[P_{n}])}\qquad\\
\le&{} \exp\left(-\frac b{b-1}\frac{\mu}{D} t\log n\left(\log^{(2)}n+\log t+\alpha+o(1)\right)\right)
\end{align*}
and
\begin{align*}
\lefteqn{P(|W_{n}-E[W_{n}]|>tE[W_{n}])}\qquad\\*
&\le{}\exp\left(-\frac b{b-1}\frac\mu D t\log n\left(\log^{(2)}n+\log t+\alpha+o(1)\right)\right)
\end{align*}
where $\mu=E[Z_1]$ and $\alpha:=\log\left(b\mu/(4D(b-1)e)\right)$.
\end{korollar}

Finally, by special choices of the edge weights and the use of transfer results in \citet{munsonius_10b}, the corresponding bounds for random linear recursive trees are obtained.
The model of linear recursive trees is introduced by \citet{pittel_94}.
Starting with the root, the linear recursive tree grows node by node.
In each step the new node is attached to a randomly chosen node of the previous ones.
The probability that node $u$ is chosen is proportional to the weight $w_u=1+\beta\deg(u)$ where $\deg(u)$ is the number of children of $u$ and $\beta\in\R_{\ge 0}$ is the parameter of the tree.
This tree model encompasses as special cases the random recursive tree ($\beta=0$) and the plane oriented recursive tree ($\beta=1$).
\begin{korollar}\label{pk_linear}
Let $P_{n}$ denote the internal path length of a random linear recursive tree of size $n$ with weight function $u\mapsto 1+(b-2)\deg(u)$ for $b\in\N$ and $b\ge 2$.
Then there exists $D>0$ such that for $t>0$ and $n\to\infty$ we have for $(P_{n}-E[P_{n}])/n$ the same tail bounds as in Theorem \ref{upper_tail_bound_wp} and in particular we have for $t>0$ and $n\to\infty$
\begin{align*}
\lefteqn{P(|P_{n}-E[P_{n}]|>tE[P_{n}])}\qquad\\
&\le{} \exp\left(-\frac 1{b-1}\frac 1D t\log n\left(\log^{(2)}n+\log t+\alpha+o(1)\right)\right)
\end{align*}
with $\alpha:=-\log\left(4D(b-1)e\right)$.
\end{korollar}

\begin{korollar}\label{wk_linear}
Let $W_{n}$ denote the Wiener index of a random linear recursive tree of size $n$ with weight function $u\mapsto 1+(b-2)\deg(u)$ for $b\in \N$ and $b\ge 2$.
Then there exists $D>0$ such that we have for $t>0$ and $n\to\infty$
\begin{align*}
\lefteqn{P(|W_{n}-E[W_{n}]|>tE[W_{n}])}\qquad\\
&\le{} \exp\left(-\frac 1{b-1}\frac 1D t\log n\left(\log^{(2)}n+\log t+\alpha+o(1)\right)\right)
\end{align*}
with $\alpha:=-\log\left(4D(b-1)e\right)$.
\end{korollar}

Using the WKB method \citet{knessl_szpankowski_99} argue for very sharp bounds for the tail of the limit distribution of the internal path length of random binary search trees.
In \citet{rueschendorf_schopp_07a} general upper bounds for tails of distributions given by a recursion of sum type are shown in the one-dimensional case.
For simply generated trees, asymptotics for the right tail of the limit distribution of the total path length and the Wiener index are shown in \citet{chassaing_janson_04} and \citet{fill_janson_09}.

This paper is organized as follows.
In section \ref{sec-2}, we consider the general recursion formula \eqref{recursion_tail} and give a proof for the upper tail bound in Theorem \ref{upper_tail_bound}.
The $b$-ary recursive tree with weighted edges is defined in section \ref{sec-3}.
We then show the stochastic domination condition in this case by a coupling argument and conclude Theorem \ref{upper_tail_bound_wp} and Corollary \ref{korollar_asymptotischer_tail} in section \ref{sec-3.1}.
Finally, we conclude by transfer results from \citet{munsonius_10b} the upper tail bounds in case of random linear recursive trees (Corollary \ref{pk_linear} and Corollary \ref{wk_linear}) in section \ref{sec-3.2}.
At the end, we give a summary of corresponding results concerning lower tail bounds for the Wiener index in section \ref{sec-4}.

We denote by $\|\cdot\|$ the Euclidean norm in $\R^k$ and by $\|\cdot\|_\mathrm{op}$ the operator norm for matrices.
Equality in distribution is written as $\eqd$.
For functions $f$ and $g$ we write $f=o(g)$, $f=O(g)$ and $f=\Theta(g)$ if $\lim_{n\to\infty}f(n)/g(n)=0$, $|f(n)/g(n)|\le C$ and $c\le |f(n)/g(n)|\le C$ for all $n$ with some constants $0<c\le C<\infty$ respectively.

\subsection*{Acknowledgment}
The author thanks Ralph Neininger for posing the problem of tail bounds and pointing some related literature out to him.

\section{Upper tail bound for a general recursion}\label{sec-2}
We consider a random $k$-dimensional vector $X_n=(\vect{X}{n,1}{n,k})$ which solves the distributional recursion formula
\[
X_n\eqd\sum_{i=1}^bA_i(I_n)X_{I_{n,i}}^{(i)}+d(I_n,Z)
\]
where $X_n^{(1)},\ldots,X_n^{(b)}$ have the same distribution as $X_n$, $d:\R^b\times\R^b\to\R^k$ and $A_i:\R^b\to\R^{k\times k}$ are deterministic functions, $Z\in\R_{\ge 0}^b$ and $I_n=(\vect{I}{n,1}{n,b})\in\{\num{0}{n-1}\}^b$ are random vectors with $E[d(I_n,Z)]=0$, and $X_n^{(1)},\ldots,X_n^{(b)}$, $I_n$, $Z$ are independent.

We denote by $\preceq_{\mathrm{st}}$ the stochastic order and by $U$ a random variable uniformly distributed in $[0,1]$.

\begin{lemma}\label{chernoff_small_s}
Let $X_n$ be a solution of the distributional recursion \eqref{recursion_tail}.
Assume that $X_1=0$, $\|d(I_n,Z)\|\le D$ almost surely for all $n\in\N$ and for a constant $D\in\R$ and that
\eqn\label{assumption_stochastic_order}
\sum_{i=1}^b\|A_i(I_n)\|_{\mathrm{op}}^2\preceq_{\mathrm{st}}1-U(1-U)
\nqe
as well as $\|A_i\|_{\mathrm{op}}\le 1$.
Let $\gamma\approx 2.0047$ be the positive solution of
\begin{align*}
\frac {12}{7}{}={}&e^{\frac 2\gamma}-\frac 2\gamma&\text{and}&&K{}={}&\frac 52D^2\gamma^2.
\end{align*}
Then we have for all $s\in\R^k$ with $\|s\|\le 1/(\gamma D)$ and for all $n\in\N$
\[
E[\exp(\skalar{s}{X_n})]\le \exp\left(K\|s\|^2\right).
\]
\end{lemma}

\begin{proof}{}~
We show the claim by induction on $n$.
For $n=1$ we have $X_1=0$ and there is nothing to show.

Using the recursion formula and the given independence we get for $n\ge 2$
\begin{align*}
\lefteqn{E\left[\exp\left(\skalar{s}{X_n}\right)\right]}\quad\\
{}={}&E\left[\exp\left(\skalar{s}{\sum_{i=1}^bA_i(I_n)X_{I_{n,i}}^{(i)}+d(I_n,Z)}\right)\right]\\
{}={}&\kern-3ex\sum_{x\in\{0,\ldots,n-1\}^b}\kern-3ex E\left[e^{\skalar{s}{d(x,Z)}}\right]\prod_{i=1}^b E\left[\exp\left(\skalar{(A_i(x))^Ts}{X_{x_i}^{(i)}}\right)\mid I_n=x\right]P(I_n=x).
\end{align*}
The assumption $\|A_i(x)\|_{\mathrm{op}}\le 1$ implies $\|A_i(x)^Ts\|\le\|s\|\|A_i(x)\|_{\mathrm{op}}\le \|s\|$.
Since for every $i\in\{\num{1}{b}\}$ we have $x_i\le n-1$ we can apply the induction hypothesis.
Therefore, we obtain
\begin{align}\label{induction_step}
\lefteqn{E\left[\exp\left(\skalar{s}{X_n}\right)\right]}\qquad\notag\\
{}\le{}&\sum_{x\in\{0,\ldots,n-1\}^b}E\left[\exp\left(\skalar{s}{d(x,Z)}\right)\right]\exp\left(\sum_{i=1}^b K\|s\|^2\|A_i(x)\|_\mathrm{op}^2\right)P(I_n=x)\notag\\
{}={}&E\left[\exp\left(\skalar{s}{d(I_n,Z)}\right)\exp\left(K\|s\|^2\sum_{i=1}^b\|A_i(I_n)\|_\mathrm{op}^2\right)\right].
\end{align}
By condition \eqref{assumption_stochastic_order} and monotonicity of $x\mapsto e^{\lambda x}$ we conclude
\begin{align}\label{alikhan_neininger}
E\left[\exp\left(\skalar{s}{X_n}\right)\right]{}\le{}&E\left[\exp\left(\skalar{s}{d(I_n,Z)}\right)\exp\left(K\|s\|^2(1-U(1-U))\right)\right]\notag\\[1.3ex]
{}={}&e^{K\|s\|^2}E\left[\exp\left(\skalar{s}{d(I_n,Z)}\right)\exp\left(-K\|s\|^2U(1-U)\right)\right].
\end{align}
Hence, using the Cauchy--Schwarz inequality it suffices to show that
\begin{align}\label{zuzeigen_tail}
\lefteqn{\left(E\left[\exp\left(\skalar{s}{d(I_n,Z)}\right)\exp\left(-K\|s\|^2 U(1-U)\right)\right]\right)^2}\qquad\notag\\[1.3ex]
{}\le{}&E\left[\exp\left(2\skalar{s}{d(I_n,Z)}\right)\right]E\left[\exp\left(-2K\|s\|^2U(1-U)\right)\right]\\[1.3ex]
{}\le{}&1\notag.
\end{align}
By assumption $\|d(I_n,Z)\|\le D$ holds almost surely and $E[d(I_n,Z)]=0$.
Thus, we get for $\|s\|\le 1/(\gamma D)$
\begin{align}\label{faktor1}
E\left[\exp\left(2\skalar{s}{d(I_n,Z)}\right)\right]{}={}&1+E\left[\skalar{s}{d(I_n,Z)}^2\sum_{k=2}^\infty\frac{2^k\skalar{s}{d(I_n,Z)}^{k-2}}{k!}\right]\notag\\
{}\le{}&1+\|s\|^2D^2\sum_{k=2}^\infty\frac{2^k}{k!\gamma^{k-2}}\notag\\
{}={}&1+\|s\|^2D^2\gamma^2\left(e^{\frac 2\gamma}-1-\frac 2\gamma\right).
\end{align}
For all $x>0$ we have
\[
e^{-x}\le 1-x+\frac {x^2}{2}.
\]
This yields for the second factor in \eqref{zuzeigen_tail}
\begin{align}\label{faktor2}
E\left[e^{-2K\|s\|^2U(1-U)}\right]{}\le{}&E\left[1-2K\|s\|^2U(1-U)+2K^2\|s\|^4U^2(1-U)^2\right]\notag\\
{}={}&1-\frac 13K\|s\|^2+\frac 1{15}K^2\|s\|^4.
\end{align}
With \eqref{faktor1} and \eqref{faktor2} we see that \eqref{zuzeigen_tail} will follow from
\[
\left(1+\|s\|^2D^2\gamma^2\left(e^{\frac 2\gamma}-1-\frac 2\gamma\right)\right)\left(1-\frac 13K\|s\|^2+\frac 1{15}K^2\|s\|^4\right)\le 1.
\]
This in turn is equivalent to $f(\|s\|)\le 0$ for
\begin{multline}
f(\|s\|):=D^2\gamma^2\left(e^{\frac 2\gamma}-1-\frac 2\gamma\right)-\frac 13K\\
-\left(\frac 13KD^2\gamma^2\left(e^{\frac 2\gamma}-1-\frac 2\gamma\right)-\frac 1{15}K^2\right)\|s\|^2\\
+\frac 1{15}K^2D^2\gamma^2\left(e^{\frac 2\gamma}-1-\frac 2\gamma\right)\|s\|^4.
\end{multline}
We substitute $K=5/2D^2\gamma^2$ and $e^{2/\gamma}-1-2/\gamma=5/7$ and obtain
\[
f(\|s\|)=D^2\gamma^2\left(\frac 57-\frac 56+\left(\frac 5{12}-\frac {25}{42}\right)(D\gamma\|s\|)^2+\frac {25}{84}(D\gamma\|s\|)^4\right).
\]
We see that $f(0)\le 0$ and $f(1/(\gamma D))=0$.
Since $f$ is a biquadratic function in $\|s\|$ with a positive coefficient corresponding to $\|s\|^4$ and $f(0)\le 0$ it has at most two real roots.
On the interval between these two roots the function is negative and outside this interval the function takes only positive values. 
Since $f(1/(\gamma D))=0$ we therefore get $f(\|s\|)\le 0$ for all $s$ with $0\le\|s\|\le 1/(\gamma D)$.
\end{proof}

\begin{lemma}\label{chernoff_large_s}
Let $X_n$ be a solution of the distributional recursion \eqref{recursion_tail}.
Assume that $X_1=0$, $\|d(I_n,Z)\|\le D$ almost surely for all $n\in\N$ and for a constant $D\in\R$ and that
\[
\sum_{i=1}^b\|A_i(I_n)\|_{\mathrm{op}}^2\preceq_{\mathrm{st}}1-U(1-U)
\]
as well as $\|A_i\|_{\mathrm{op}}\le 1$.
Let $\gamma\approx 2.0047$ be the positive solution of $\frac {12}{7}=e^{\frac 2\gamma}-\frac 2\gamma$ and $L_0\approx 5.0177$ be the largest root of $e^L=6L^2$.
Then we have for $1/(\gamma D)\le\|s\|\le L$
\[
E\left[\exp\left(\skalar{s}{X_n}\right)\right]\le \exp\left(K_L\|s\|^2\right)
\]
where
\[
K_L:=\begin{cases}24D^2,&\text{ for $1/(\gamma D)< L\le L_0/D$,}\\4\frac {1}{L^2}e^{LD},&\text{ for $L_0/D< L$.}\end{cases}
\]
\end{lemma}

\begin{proof}{}
We again use induction on $n$.
For $n=1$ there is nothing to show.
We use the same arguments as in the beginning of the proof of Lemma \ref{chernoff_small_s} and get \eqref{alikhan_neininger}:
\[
E\left[e^{\skalar{s}{X_n}}\right]\le e^{K_L\|s\|^2}E\left[\exp\left(\skalar{s}{d(I_n,Z)}\right)\exp\left(-K_L\|s\|^2U(1-U)\right)\right]
\]
for a random variable $U$ which is uniformly distributed on $[0,1]$.
Hence, it suffices to prove \eqref{zuzeigen_tail} under the new assumptions.
Since $\|d(I_n,Z)\|\le D$ almost surely the proof is completed by showing
\[
e^{D\|s\|}E\left[e^{-K_L\|s\|^2U(1-U)}\right]\le 1.
\]
\citet[Section 4]{fill_janson_01} proved that for any $K>0$
\eqn\label{fill_janson_1}
E\left[\exp\left(-2K\|s\|^2U(1-U)\right)\right]\le\frac{1-\exp\left(-K\frac{\|s\|^2}{2}\right)}{K\frac{\|s\|^2}{2}}
\nqe
and for $0.42\le|\lambda|\le M$
\begin{align}\label{fill_janson_2}
e^{|\lambda|}\frac{1-\exp\left(-K_M\frac{\lambda^2}{2}\right)}{K_M\frac{\lambda^2}{2}}{}\le{}& 1&\text{when}&&K_M{}={}&\begin{cases}12,&\text{ for $M\le L_0$,}\\2e^M/M^2,&\text{ for $L_0<M$}.\end{cases}
\end{align}
In the present situation, it follows
\begin{align*}
e^{D\|s\|}E\left[e^{-K_L\|s\|^2U(1-U)}\right]{}\le{}& e^{D\|s\|}\frac{1-\exp\left(-\frac{K_L}{2D^2}\frac{D^2\|s\|^2}{2}\right)}{\frac{K_L}{2D^2}\frac{D^2\|s\|^2}{2}}\\
\le{}&1
\end{align*}
when $1/\gamma\le D\|s\|\le LD$ and
\[
K_L=\begin{cases}24D^2,&\text{ for $L\le L_0/D$,}\\4\frac 1{L^2}e^{LD},&\text{ for $L_0/D< L$.}\end{cases}
\]
Thus, we obtain the claim because it is $1/\gamma \ge 0.42$.
\end{proof}

We summarize the results of the two preceding lemmas.
\begin{korollar}\label{kor_tail}
Let $X_n$ be a solution of the distributional recursion \eqref{recursion_tail}.
Assume that $X_1=0$, $\|d(I_n,Z)\|\le D$ almost surely for all $n\in\N$ and for a constant $D\in\R$ and that
\[
\sum_{i=1}^b\|A_i(I_n)\|_{\mathrm{op}}^2\preceq_{\mathrm{st}}1-U(1-U)
\]
as well as $\|A_i\|_{\mathrm{op}}\le 1$.
Let $\gamma\approx 2.0047$ be the positive solution of $12/7=e^{2/\gamma}-2/\gamma$ and $L_0\approx 5.0177$ be the largest root of $e^L=6L^2$.
Then we have for every $s$ and $n\ge 1$
\[
E\left[\exp\left(\skalar{s}{X_n}\right)\right]\le\begin{cases}\exp\left(\frac 52\gamma^2D^2\|s\|^2\right),&\text{ for $0\le\|s\|\le 1/(\gamma D)$},\\\exp\left(24D^2\|s\|^2\right),&\text{ for $1/(\gamma D)<\|s\|\le L_0/D$},\\\exp\left(4e^{D\|s\|}\right),&\text{ for $L_0/D<\|s\|$}.\end{cases}
\]
\end{korollar}

\begin{proof}{}
The bounds for $\|s\|\le L_0/D$ follow immediately from Lemma \ref{chernoff_small_s} and Lemma \ref{chernoff_large_s}.
Since the function $x\mapsto e^{Dx}/x^2$ is monotonically increasing on the interval $[L_0/D,\infty)$, Lemma \ref{chernoff_large_s} yields also the bound in the case $\|s\|> L_0/D$.
\end{proof}

Now, we get the tail bound for any entry of the vector $X_n$.

\begin{proof}{ of Theorem \ref{upper_tail_bound}}
We denote by $e_j$ the vector with $1$ in the $j$-th entry and $0$ elsewhere.
We use Chernoff's bounding technique and obtain for $u>0$ and $j\in\{\num{1}{k}\}$ with Corollary \ref{kor_tail}
\begin{align*}
P\left(X_{n,j}>t\right){}={}&P\left(\exp\left(uX_{n,j}\right)>\exp(ut)\right)\\
{}\le{}&E\left[\exp\left(uX_{n,j}-ut\right)\right]\\
{}={}&E\left[\exp\left(u\skalar{e_j}{X_n}-ut\right)\right]\\
{}\le{}&\exp\left(K_uu^2-ut\right),
\end{align*}
where
\[
K_u=\begin{cases}\frac 52\gamma^2D^2,&\text{ for $0\le u\le 1/(\gamma D)$},\\24D^2,&\text{ for $1/(\gamma D)<u\le L_0/D$},\\4\frac{e^{Du}}{u^2},&\text{ for $L_0/D<u$}.\end{cases}
\]
For the left tail we receive analogously
\begin{align*}
P\left(X_{n,j}<-t\right){}={}&P\left(\exp\left(uX_{n,j}\right)<\exp(-ut)\right)\\
{}\le{}&E\left[\exp\left(-uX_{n,j}-ut\right)\right]\\
{}={}&E\left[\exp\left(-u\skalar{e_j}{X_n}-ut\right)\right]\\
{}\le{}&\exp\left(K_uu^2-ut\right).
\end{align*}
In order to minimize this bound we are looking for the minimum of the function $f(u):=K_uu^2-ut$.
This function takes its minimum at $\check u_i(t)$ and has the value $f(\check u_i(t))$ for
\begin{align*}
\check u_1(t){}={}&\frac t{5D^2\gamma^2},&f(\check u_1(t))&{}={}-\frac{t^2}{10\gamma ^2D^2}&&\text{ for }K_u=\frac 52D^2\gamma^2,\\
\check u_2(t){}={}&\frac t{48D^2},&f(\check u_2(t))&{}={}-\frac{t^2}{96 D^2}&&\text{ for }K_u=24D^2,\\
\check u_3(t){}={}&\frac 1D\log\frac t{4D},&f(\check u_3(t))&{}={}\frac tD-\frac tD\log\frac t{4D}&&\text{ for }K_u=4\frac{e^{Du}}{u^2}
\end{align*}
where $\check u_i(t)\in U_i$ with $U_1:=[0,1/(\gamma D)]$, $U_2:=(1/(\gamma D),L_0/D]$ and $U_3:=(L_0/D,\infty)$.

If $\check u_i(t)\not\in U_i$ for a given $t$, we can take $u$ at the proper boundary of $U_i$.

Comparing the different values of the minimum for $i=1,2,3$ we obtain the total minimum.
For $t\in[0,5\gamma D]$ we have the following possibilities:
\begin{align*}
\check u_1(t){}={}&\frac t{5D^2\gamma^2},&f(\check u_1(t))&{}={}-\frac{t^2}{10\gamma ^2D^2}\\
u_2{}={}&\frac 1{\gamma D},&f(u_2)&{}={}\frac{24}{\gamma^2}-\frac t{\gamma D}\\
u_3{}={}&\frac{L_0}{D},&f(u_3)&{}={}4e^{L_0}-\frac{L_0}{D}t.
\end{align*}
The minimum is given for $\check u_1(t)$.

Similarly, we obtain the minimum in the other cases by making the following choices:\\
for $t\in[5\gamma D,48D/\gamma +D\sqrt{48}\sqrt{48/\gamma^2-5})$
\begin{align*}
u_1{}&{}=\frac 1{\gamma D},&K_u{}&{}=\frac 52\gamma^2D^2,&f(u_1){}&=\frac 52-\frac t{\gamma D},
\end{align*}
for $t\in[48D/\gamma +D\sqrt{48}\sqrt{48/\gamma^2-5},48 DL_0)$
\begin{align*}
\check u_2(t){}&{}=\frac t{48D^2},&K_u{}&{}=24D^2,&f(\check u_2(t)){}&=-\frac{t^2}{96D^2},
\end{align*}
for $t\in[48 DL_0,4De^{L_0})$
\begin{align*}
u_2{}&{}=\frac {L_0}{D},&K_u{}&{}=24D^2,&f(u_2){}&=24L_0^2-\frac{L_0}{D}t,
\end{align*}
and for $t\in[4De^{L_0},\infty)$
\begin{align*}
\check u_3(t){}&{}=\frac 1D\log\frac{t}{4D},&K_{\check u_3(t)}{}&{}=\frac{4e^{D\check u_3(t)}}{{\check u_3(t)}^2},&f(\check u_3(t)){}&=\frac tD-\frac tD\log\frac{t}{4D}.
\end{align*}
\end{proof}

\section{Applications to random trees}\label{sec-3}
An example of a vector which satisfies the recursion formula \eqref{recursion_tail} is the vector consisting of the internal path length and the Wiener index of a random tree in which all subtrees are (conditioned upon their sizes) an independent copy of the whole tree.

The internal path length of a rooted tree is the sum of all node depths of the tree.
The depth of a node is given by the number of edges on the path from the node to the root.
Analogously, the Wiener index is the sum of the distances between all unordered pairs of nodes where the distance is given by the number of edges on the unique path between the two nodes.

\subsection{The random $b$-ary recursive trees with weighted edges}\label{sec-3.1}
In this section we consider the special case of a random $b$-ary recursive tree with weighted edges.

The random $b$-ary recursive tree is a rooted, ordered, labelled tree where the outdegree is bounded by $b$ and the labels along each path beginning at the root increase.
We define this tree model by the following recursive procedure.
We consider the infinite complete $b$-ary rooted, ordered tree and start with the root as the first internal node and its $b$ children as external nodes.
Given \emph{the random $b$-ary recursive tree} with $n$ internal nodes, the $n+1$st internal node is added in the following way.
We choose a random node uniformly distributed on the set of all current external nodes, change it to an internal one and add the $b$ children of this new node to the set of external nodes.
Finally, the nodes are labelled in the order of their appearance.

Let $Z:=(\vect{Z}{1}{b})\in\R_{\ge 0}^b$ be a random vector with non-negative entries and attach to every node $u$ of the complete infinite $b$-ary tree an independent copy $Z^{(u)}$ of $Z$.
We consider the entries of $Z^{(u)}$ as weights of the edges from $u$ to its $b$ children.
If all $Z^{(u)}$ are independent of $T_n$, we refer to $T_n$ supplied with the family $\{Z^{(u)}\}$ as a \emph{random $b$-ary recursive tree with edge weights $Z$}.

While the entries of the vector $Z$ may depend on each other, we assume that they are identically distributed, i.e.\ for all $i,j\in\{\num{1}{b}\}$ we have $Z_i\eqd Z_j$, and denote its expectation by $\mu:=E[Z_1]$.
This assumption is not restrictive for the intended limit theorems as can be seen by a permutation argument \citep[see][p.\ 14--15]{munsonius_10}.
For instance, the shape of the random binary search tree is equally distributed as the shape of the random $b$-ary recursive tree with egde weights $(Z_1,Z_2)=(1,1)$ for $b=2$.

Let $Y_n=(W_n,P_n)$ denote the vector consisting of the Wiener index and the internal path length of the random $b$-ary recursive tree of size $n$ with edge weights $Z$.
In \citet{munsonius_10b} it is shown that the vector
\eqn\label{rec_pkn}
X_{n}:=\begin{bmatrix} \frac 1{n^2}&0\\0&\frac 1n\end{bmatrix} \left(Y_{n}-E[Y_{n}]\right)
\nqe
satisfies the recursion formula \eqref{recursion_tail} where the matrices $A_i(I_n)$ are given by
\[
A_i(I_n)=\begin{bmatrix} \frac {I_{n,i}^2}{n^2}&\frac{I_{n,i}(n-I_{n,i})}{n^2}\\[1.2ex]0&\frac {I_{n,i}}{n}\end{bmatrix}
\]
and the vector $d(I_n,Z)$ is given by
\begin{equation}\label{toll-term-a}
d_1^{(n)}{}={}\frac b{b-1}\mu \sum_{i=1}^b\frac{I_{n,i}}{n}\log \frac{I_{n,i}}{n}+\sum_{i\not= j}\left(\frac 12(Z_i+Z_j)+\frac{b}{b-1}\mu \right)\frac{I_{n,i}}{n}\frac{I_{n,j}}{n}+o(1)
\end{equation}
and
\begin{equation}\label{toll-term-b}
d_2^{(n)}=\frac b{b-1}\mu\sum_{i=1}^b \frac{I_{n,i}}{n}\log \frac{I_{n,i}}{n}+\sum_{i=1}^b Z_i\frac{I_{n,i}}{n}+o(1).
\end{equation}
To apply the result of the previous section, we have to prove the stochastic domination condition for the $b$-ary recursive tree.

\subsubsection{Coupling}
For $(\vect{x}{1}{b})\in\R^b$ we denote by $(\vect{x}{(1)}{(b)})$ the order statistic, i.e.\
\[
x_{(1)}\ge x_{(2)}\ge\cdots\ge x_{(b)}
\]
and the entries of $(\vect{x}{1}{b})\in\R^b$ and $(\vect{x}{(1)}{(b)})$ are the same.
We consider the space $\R^b$ with the partial order given by
\[
(\vect{x}{1}{b})\le (\vect{y}{1}{b})\qquad:\Longleftrightarrow\qquad x_i\le y_i\quad\text{for all $i\in\{\num{1}{b}\}$}
\]
and define $E_b:=\{(\vect{x}{1}{b})\in\R^b_{\ge 0}\mid x_1\ge x_2\ge\ldots\ge x_b\}$.
Moreover, we denote by $PU(b)$ a P\'olya urn with balls of $b$ different colors $\{1,\ldots,b\}$, which contains at the beginning $1$ ball of each color and after a ball of color $j$ is drawn, it is returned to the urn together with another $b-1$ balls of the same color.

Considering the evolution process which yields the random $b$-ary recursive tree, it is not difficult to see, that the vector of the  sizes of the subtrees has the same distribution as the vector of the numbers of drawings of a ball of the different colors in the urn described above \citep[for more details see][Section 2.2]{munsonius_10}.
Using this, the next two lemmas provide the estimate we need.

\begin{lemma}\label{urn_coupling}\hspace{-0.6em}
For $j\in\{\num{1}{b}\}$ let $J_{n,j}$ denote the number of times that the drawn ball is of the color $j$ during the first $n$ drawings of the P\'olya urn $PU(b)$ and $I_{n,j}$ the corresponding size for the P\'olya urn $PU(b+1)$.
Then we have for the vectors $J_n:=(\vect{J}{n,1}{n,b})$ and $I_n:=(\vect{I}{n,1}{n,b+1})$
\[
(\vect{I}{n,(1)}{n,(b)})\preceq_{\text{st}}(\vect{J}{n,(1)}{n,(b)}).
\]
\end{lemma}

We prove this lemma by using a result about stochastic domination between Markov chains \citep[see][Section IV.5, Theorem (5.8)]{lindvall_92}.
With $e_j\in\R^b$ we denote the vector where all entries are $0$ except the $j$-th entry which is $1$.
\begin{proof}{}
It suffices to show that there is a coupling of $I'_n:=(\vect{I}{n,(1)}{n,(b)})$ and $J'_n:=(\vect{J}{n,(1)}{n,(b)})$ such that $I'_n\le J'_n$ almost surely.

The sequence $J'_n$ resp.\ $I'_n$ is a Markov chain.
To write down the transition probabilities we define $\alpha_j:E_b\to \N_0$ by
\[
\alpha_j(\vect{x}{1}{b}):=\begin{cases}|\{i\mid x_j=x_i\}|,&\text{if $x_{j-1}>x_j$,}\\0,&\text{otherwise}.\end{cases}
\]
Thus, the transition probability for $J'_n$ is given by the kernel $K_n:E_b\times E_b\to[0,1]$ with
\[
K_n(x,x+e_j):=P(J'_{n+1}=x+e_j\mid J'_n=x)=\frac{1+x_j(b-1)}{b+n(b-1)}\alpha_j(x)
\]
for $x=(\vect{x}{1}{b})\in E_b$ and $j=\num{1}{b}$.
For the transition probability of $I'_n$ we get the kernel $K'_n:E_b\times E_b\to[0,1]$ with
\[
K'_n(x,x+e_j):=P(I'_{n+1}=x+e_j\mid I'_n=x)=\frac{1+x_jb}{b+1+nb}\alpha_j(x)
\]
for $j=\num{1}{b}$ and 
\[
K'_n(x,x):=P(I'_{n+1}=x\mid I'_n=x)=\frac{1+\left(n-\sum_{i=1}^bx_i\right)b}{b+1+nb}.
\]

Let $x,y\in E_b$ with $y\le x$. We claim that $K'_n(y,\cdot)$ is stochastically dominated by $K_n(x,\cdot)$.

If
\[
P(I'_{n+1}=y+e_j\mid I'_n=y)>0
\]
we have $\alpha_j(y)\not=0$.
For $y_j<x_j$ we get $y+e_j\le x$.
Thus, we only have to consider the case where $\alpha_j(y)\not=0$ and $y_j=x_j$.
Let $\vect{j}{1}{m}$ be the components for which $\alpha_{j_l}(y)\not=0$ and $x_{j_l}=y_{j_l}$ for $1\le l\le m$.
Then we have $\alpha_{j_l}(x)\ge\alpha_{j_l}(y)$ because $x_{j_l-1}\ge y_{j_l-1}> y_{j_l}=x_{j_l}$.
Since $x_{j_l}\le n$ we get
\[
\frac{1+x_{j_l}(b-1)}{b+n(b-1)}\ge\frac{1+x_{j_l}b}{b+1+nb}.
\]
This yields for all $l\in\{\num{1}{m}\}$
\[
K_n(x,x+e_{j_l})=\frac{1+x_{j_l}(b-1)}{b+n(b-1)}\alpha_{j_l}(x)\ge \frac{1+x_{j_l}b}{b+1+nb}\alpha_{j_l}(y)=K'_n(y,y+e_{j_l}).
\]
For $i\in\{\num{1}{b}\}\setminus\{\vect{j}{1}{m}\}$ we obviously have
\[
K_n(x,x+e_i)\ge 0= K_n'(y,y+e_i).
\]
Hence, we can find a coupling $(\tilde J_{n+1},\tilde I_{n+1})$ of $(J'_{n+1},I'_{n+1})$ with
\[
P\big(\tilde I_{n+1}\le \tilde J_{n+1}\mid(\tilde J_n,\tilde I_n)=(x,y)\big)=1
\]
for all $x,y\in E_b$ with $y\le x$. 
This implies that $K_n(x,\cdot)$ dominates stochastically $K'_n(y,\cdot)$.
Because of the Markov property we conclude with \citet[Section IV.5, Theorem (5.8)]{lindvall_92} that there exists a coupling $(\tilde J,\tilde I)$ of $J'$ and $I'$, such that $\tilde I_n\le\tilde J_n$ almost surely for all $n\in\N$.
\end{proof}

\begin{lemma}\label{abschaetzung_ana}
Let $f:[0,1]\to\R$ be the function given by 
\[
f(x)=x^4+(x^2-x^3)\left(1+\sqrt{x^2+1}\right).
\]
Then, for $(\vect{x}{1}{b})\in E_b$ and $(\vect{y}{1}{b+1})\in E_{b+1}$ with 
$
\sum_{i=1}^bx_i=\sum_{j=1}^{b+1}y_j=1
$
and $(\vect{y}{1}{b})\le (\vect{x}{1}{b})$ we have
\[
\sum_{i=1}^{b+1}f(y_i)\le\sum_{i=1}^bf(x_i).
\]
\end{lemma}

\begin{proof}{}
We first show that the function $f$ is convex.
To do this, we derive the second derivative which is given by
\begin{align*}
f''(x){}={}&12x^2+(2-6x)(1+\sqrt{x^2+1})+\frac {2(2x^2-3x^3)}{\sqrt{x^2+1}}+\frac {x^2-x^3}{(x^2+1)^{\frac 32}}\\
\ge{}&10x^2+(2-6x)(1+\sqrt{x^2+1})+\frac {2(3x^2-3x^3)}{\sqrt{x^2+1}}+\frac {x^2-x^3}{(x^2+1)^{\frac 32}}.
\end{align*}
To show convexity it suffices to show $f''\ge 0$.
Since for $x\in[0,1]$ it is $x^2\ge x^3$ it remains to show
\[
g(x):=10x^2+(2-6x)(1+\sqrt{x^2+1})\ge 0.
\]
By consideration of the first and second derivatives we see that the minimum of $g$ is obtained by $x=3/4$ with $g(3/4)=0$. 
Taking everything into account, we obtain $f''(x)\ge 0$ for all $x\in[0,1]$ which implies that the first derivative $f'$ is monotone increasing.

By assumption, there exist numbers $\vect{\alpha}{1}{b}\ge 0$ with $x_i=y_i+\alpha_iy_{b+1}$ for $i=\num{1}{b}$ and $\sum_{i=1}^b\alpha_i=1$.
The monotonicity of $f'$ and $f(0)=0$ imply with the mean value theorem
\[
\frac{f(x_i)-f(y_i)}{\alpha_iy_{b+1}}=f'(\xi)\ge f'(\eta)=\frac{f(y_{b+1})}{y_{b+1}}
\]
for some $\xi\in[y_i,x_i]$ and $\eta\in[0,y_{b+1}]\subset[0,y_i]$.
This finally yields
\[
\sum_{i=1}^bf(x_i)\ge\sum_{i=1}^{b}(f(y_i)+\alpha_if(y_{b+1}))=\sum_{i=1}^{b+1}f(y_i).
\]
\end{proof}

\begin{proof}{ of Theorem \ref{upper_tail_bound_wp}}
As seen in equation \eqref{rec_pkn} we have for the vector $X_n$ the recursion formula \eqref{recursion_tail} where 
\[
A_i(I_n)=\begin{bmatrix} \frac {I_{n,i}^2}{n^2}&\frac{I_{n,i}(n-I_{n,i})}{n^2}\\[1.2ex]0&\frac {I_{n,i}}{n}\end{bmatrix}.
\]
For the operator norm we obtain
\[
\|A_i(I_n)\|_\mathrm{op}^2=\|A_i^T(I_n)A_i(I_n)\|_\mathrm{op}.
\]
The matrix $A_i^T(I_{n})A_i(I_n)$is symmetric.
Thus, its operator norm is given by the largest absolute eigenvalue.
Solving the characteristic equation for the matrix we obtain that its eigenvalue being larger in absolute value is given by
\[
\frac{I_{n,i}^2}{n^2}\left(1-\frac{I_{n,i}}{n}+\frac{I_{n,i}^2}{n^2}+\left(1-\frac{I_{n,i}}{n}\right)\sqrt{\frac{I_{n,i}^2}{n^2}+1}\right)=f\left(\frac{I_{n,i}}{n}\right)
\]
with the function $f$ as given in Lemma \ref{abschaetzung_ana}.
This yields with Lemma \ref{abschaetzung_ana} and Lemma \ref{urn_coupling}
\[
\sum_{i=1}^b\|A_i(I_n)\|_{\mathrm{op}}^2\preceq_{\mathrm{st}}\|A_1(J_{n})\|_\mathrm{op}^2+\|A_2(J_{n})\|_\mathrm{op}^2
\]
where $J_n=(J_{n,1},J_{n,2})$ is the vector of the sizes of the subtrees of a random binary search tree, i.e.\ $J_{n,1}$ is uniformly distributed on $\{\num{0}{n-1}\}$ and $J_{n,2}=n-1-J_{n,1}$.
By Lemma 2.2 from \citet{alikhan_neininger_07} we get the stochastic domination condition
\[
\sum_{i=1}^b\|A_i(I_n)\|_{\mathrm{op}}^2\preceq_{\mathrm{st}}1-U(1-U).
\]

Considering the toll vector $d(I_n,Z)$ in \eqref{toll-term-a} and \eqref{toll-term-b}, the boundedness of 
$\|Z\|$ implies that its norm is bounded almost surely by some constant $D$.
Furthermore, we trivially have $X_1=0$ since the tree with one node is only the root.
The claim follows by Theorem \ref{upper_tail_bound}.
\end{proof}
In \citet{munsonius_10b} the asymptotic expansion of the expectation of $P_{n}$ and $W_{n}$ is determined.

Using these results, we obtain asymptotic tail bounds.
\begin{proof}{ of Corollary \ref{korollar_asymptotischer_tail}}
With $y_n=\frac{t\E[P_{n}]}{n}=\frac{b}{b-1}\mu t\log n+O(1)$ we obtain by Theorem \ref{upper_tail_bound_wp} because of $\lim_{n\to\infty}y_n=\infty$,
\begin{align*}
\lefteqn{P(|P_n-E[P_n]|\ge tE[P_{n}])}\\
{}={}&P\left(\frac{|P_{n}-E[P_{n}]|}{n}\ge y_n\right)\\
{}\le{}&\exp\left(\left(\frac b{b-1} \frac\mu D t\log n+O(1)\right)\left(1-\log\frac{tb\mu\log n}{4D(b-1)}\right)\right)\\
=&\exp\left(-\frac b{b-1}\frac\mu D t\log n\left(\log^{(2)}n+\log t+\alpha+o(1)\right)\right).
\end{align*}
With $z_n=\frac{t\E[W_{n}]}{n^2}=y_n+O(1)$ the claim for $W_{n}$ follows as well.
\end{proof}

\subsection{Random linear recursive trees}\label{sec-3.2}
In this section we transfer the results for the random $b$-ary recursive tree with weighted edges to linear recursive trees.
In this tree, every node $u$ has a weight $w_u$.
Starting with the root, the tree grows node by node.
In each step the new node is attached to a randomly chosen node of the previous ones.
The probability that node $u$ is chosen is proportional to the weight $w_u$ of the node.
In the case of linear recursive trees the weight is given by $w_u=1+\beta\deg(u)$ where $\deg(u)$ is the number of children of $u$ and $\beta\in\R_{\ge 0}$ is the parameter of the tree.

Given a random linear recursive tree $T_n$ of size $n$ with weight function $u\mapsto 1+(b-2)\deg(u)$ we consider a $b$-ary recursive tree $\tilde T_{n-1}$ of size $n-1$ where the edges are weighted by the random vector $Z$ which is obtained by a uniformly distributed permutation of the entries of $(1,0,\ldots,0)\in\R^b$.
In particular, we have $\mu=1/b$.
Denote by $\tilde P_{n-1}$ and $\tilde W_{n-1}$ (resp. $P_{n}$ and $W_{n}$) the internal path length and the Wiener index of $\tilde T_{n-1}$ (resp. $T_n$).

\begin{proof}{ of Corollary \ref{pk_linear}}
With the notation above, it is shown in \citet{munsonius_10b} that 
\[
P_{n}\eqd\tilde P_{n-1}+n-1
\]
holds.
Therefore, we have
\[
\frac{P_{n}-E[P_{n}]}{n}\eqd\frac{\tilde P_{n-1}-E[\tilde P_{n-1}]}{n}.
\]
The claim follows immediately by Theorem \ref{upper_tail_bound_wp} and Corollary \ref{korollar_asymptotischer_tail}.
\end{proof}

\begin{proof}{ of Corollary \ref{wk_linear}}
With the notation above, it is shown in \citet{munsonius_10b} that 
\[
W_{n}\eqd\tilde W_{n-1}-\tilde P_{n-1}+(n-1)^2
\]
holds.
This yields
\begin{align*}
\lefteqn{P(|W_{n}-E[W_{n}]|>tE[W_{n}])}\\
&={}P(|\tilde W_{n-1}- E[\tilde W_{n-1}]-\tilde P_{n-1}+E[\tilde P_{n-1}]|>tE[W_{n-1}]).
\end{align*}
Moreover, in \citet{munsonius_10b} it is shown that $\Var(\tilde P_{n-1})=\Theta(n^{2})$ for $n\to\infty$.
Applying Chebycheff's inequality, we obtain for any $\epsi>0$
\[
P\left(\frac{|\tilde P_{k-1,n}-E[\tilde P_{n-1}]|}{n^2}>\epsi\right)\le O\left(\frac 1{n^2}\right).
\]
Since $E[\tilde W_{n-1}]=E[W_{n}]+O(n\log n)=\Theta(n^2\log n)$ this yields with Corollary \ref{korollar_asymptotischer_tail}
\begin{align*}
\lefteqn{P(|W_{n}-E[W_{n}]|>tE[W_{n}])}\qquad\\
{}\le{}&{}P\left(\frac{|\tilde W_{n-1}-E[\tilde W_{n-1}]|}{n^2}+\epsi>\frac{tE[W_{n}]}{n^2}\right)+O\left(\frac 1{n^2}\right)\\
\le{}&{}\exp\left(-\frac 1{b-1}\frac 1D t\log n\left(\log^{(2)}n+\log t+\alpha+o(1)\right)\right).\qed
\end{align*}
\renewcommand{\qed}{}
\end{proof}

\begin{bem}\label{bem_port}
Random plane oriented recursive trees without the order of the nodes equal in distribution the random linear recursive tree with parameter $\beta=1$.
Since the internal path length as well as the Wiener index are invariant under changing the order of the tree the tail bounds in Corollary \ref{pk_linear} and Corollary \ref{wk_linear} with $b=3$ provides in particular the corresponding tail bounds for the plane oriented recursive tree.
\end{bem}

\section{Lower tail bounds for the Wiener index}\label{sec-4}
For the number of comparisons made by quicksort a lower bound for the tail is proved in \citet{mcdiarmid_hayward_96}.
There, a set of binary trees is constructed that has high probability and implies a large number of comparisons.
They succeeded in finding lower and upper bounds which have the same asymptotical behavior.
This idea is employed by \citet{alikhan_neininger_07} to prove a lower tail bound for the Wiener index of binary search trees.

In \citet[Section 7.2]{munsonius_10} the construction from \citet{alikhan_neininger_07} is extended to random $b$-ary recursive trees with weighted edges where at least one entry of $Z$ is $1$.
This yields the following lower bound on the tail of the distribution of the Wiener index.
\begin{satz}\label{lower_tail_w}
Let $W_n$ denote the Wiener index of a random $b$-ary recursive tree of size $n$ with edge weights $Z$ where $\{\vect{Z}{1}{b}\}\cap\{1\}\not=\emptyset$ and $Z_i\ge 0$.
Then we have for fixed $t>0$ and $n\to\infty$
\begin{align*}
\lefteqn{P\left(|W_n-E[W_n]|>tE[W_n]\right)}\qquad\\
&{}\ge \,\exp\left(-4\frac b{b-1}\mu  t\log n\left(\log^{(2)}n+O(\log^{(3)}n)\right)\right).
\end{align*}
\end{satz}

With the transfer results already used in section \ref{sec-3.2} we obtain a lower bound for the distribution of the Wiener index of random linear recursive trees.
\begin{satz}
Let $W_n$ denote the Wiener index of a random linear recursive tree of size $n$ with weight function $u\mapsto 1+(b-2)\deg(u)$ for $b\in\N$ and $b\ge 2$.
Then we have for fixed $t>0$ and $n\to\infty$
\begin{align*}
\lefteqn{P\left(|W_{n+1}-E[W_{n+1}]|>tE[W_{n+1}]\right)}\qquad\\
&{}\ge\,\exp\left(-4\frac 1{b-1} t\log n\left(\log^{(2)}n+O(\log^{(3)}n)\right)\right).
\end{align*}
\end{satz}

Remark \ref{bem_port} holds true also for the lower tail bound.
\begin{bem}
The constants $D$ arising in the results depend on the specific toll function which in turn depends on the functional and the tree model considered.
Since the toll function in \eqref{toll-term-a} and \eqref{toll-term-b} is only known up to a $o(1)$-term, it is in general not possible to determine this constant.
Nevertheless, it is an analytical problem and should be solvable for special functionals and tree models.
For instance, in the case of the vector $(W_n,P_n)$ of the binary search tree, \citet{alikhan_neininger_07} showed $D\le 1$.
\end{bem}
\bibliographystyle{plainnat}
\bibliography{../bib,../../my-bib}
\end{document}